\documentclass[12pt,reqno]{amsart}
\usepackage{enumerate, latexsym, amsmath, amsfonts, amssymb, amsthm, graphicx, color}
\usepackage{xcolor}
\usepackage{colortbl,booktabs}
\def\Z{\Bbb Z}

\def\F{\Bbb F}

\def\bg{\bigg}
\def\({\bg(}
\def\){\bg)}

\def\ord{{\rm ord}}

\def\ve{\varepsilon}

\theoremstyle{plain}
\newtheorem{theorem}{Theorem}

\newtheorem{lemma}{Lemma}
\newtheorem{corollary}{Corollary}
\newtheorem{conjecture}{Conjecture}
\theoremstyle{definition}
\newtheorem*{acknowledgment}{Acknowledgment}
\theoremstyle{remark}

\makeatletter
\@namedef{subjclassname@2010}{%
  \textup{2010} Mathematics Subject Classification}
\makeatother
 \vspace{4mm}

\begin{document}
 \baselineskip=17pt
\hbox{} {}
\medskip
\title[Primitive elements and $k$-th powers in finite fields]
{Primitive elements and $k$-th powers in finite fields}
\date{}
\author[Hai-Liang Wu and Yue-Feng She] {Hai-Liang Wu and Yue-Feng She}

\thanks{2020 {\it Mathematics Subject Classification}.
Primary 11T30; Secondary 11T06, 11A15.
\newline\indent {\it Keywords}. Primitive elements, quadratic polynomial, finite fields.
\newline \indent Supported by the National Natural Science
Foundation of China (Grant No. 11971222).}

\address {(Hai-Liang Wu) School of Science, Nanjing University of Posts and Telecommunications,
Nanjing 210023, People's Republic of China}
\email{\tt whl.math@smail.nju.edu.cn}

\address {(Yue-Feng She) Department of Mathematics, Nanjing
University, Nanjing 210093, People's Republic of China}
\email{{\tt she.math@smail.nju.edu.cn}}

\begin{abstract}
Let $\mathbb{F}_q$ be the finite field of $q$ elements, and let $k\mid q-1$ be a positive integer.
Let $f(x)=ax^2+bx+c$ be a quadratic polynomial in $\F_q[x]$ with $b^2-4ac\ne0$. In this paper, we
show that if $q>\max\{e^{e^3},(2k)^6\}$, then there is a primitive element $g$ of $\mathbb{F}_q$
such that $f(g)\in\F_q^{\times k}=\{x^k: x\in\mathbb{F}_q\setminus\{0\}\}$. Moreover, we shall
confirm a conjecture posed by Sun.
\end{abstract}

\maketitle

\section{Introduction}
\setcounter{lemma}{0}
\setcounter{theorem}{0}
\setcounter{corollary}{0}
\setcounter{remark}{0}
\setcounter{equation}{0}
\setcounter{conjecture}{0}
Let $\F_q$ be the finite field of $q$ elements, and let $\F_q^{\times}$ be the multiplicative group of
all non-zero elements over $\F_q$. For an element $g\in\F_q^{\times}$, we say that $g$ is a {\it
primitive element} of $\F_q$ if $g$ is a generator of the cyclic group $\F_q^{\times}$.

The properties of primitive elements have been extensively investigated. For example, Cohen, e Silva
and Trudgian \cite{CET2015} showed that if $q>169$ then there always exists an element
$x\in\F_q^{\times}$ such that $x$, $x+1$ and $x+2$ are all primitive elements. In 2018, Booker, Cohen,
Sutherland and Trudgian \cite{BCST2019} considered the values of quadratic polynomials at primitive
elements. In fact, let $f(x)=ax^2+bx+c$
be a quadratic polynomial in $\F_q[x]$ with $b^2-4ac\ne0$. Then they showed that there always is a
primitive element $g$ such that $f(g)$ is also a primitive element whenever $q>211$. In view of the
above, investigating values of a quadratic polynomial at primitive elements is a meaningful topic in
number theory. Recently, Sun \cite{Sun} posed the following conjecture.
\begin{conjecture}\label{Conj. A}{\rm (Sun)}
Let $p$ be a prime, and let $f(x)=ax^2+bx+c$ be a quadratic polynomial in $F_p[x]$ with $b^2-4ac\ne0$.
Suppose that $p>13$ and $p\ne 19,31$. Then there is a primitive element $g$ of $\F_p$ such that $f(g)$
is a non-zero square in $\F_p$.
\end{conjecture}

Motivated by the above work, let $k$ be a positive integer with $k\mid q-1$, and let $\F_q^{\times
k}=\{x^k: x\in\F_q^{\times}\}$. Let $f(x)=ax^2+bx+c$ be a quadratic polynomial in $\F_q[x]$ with
$b^2-4ac\ne0$. In this paper we mainly consider the following problem:

Is there some positive integer $q_0(k)$ such that $f(g)\in\F_q^{\times k}$ for some primitive element
$g$ of $\F_q$ whenever $q>q_0(k)$?

Now we state our first result.
\begin{theorem}\label{Thm. A}
Let $\F_q$ be the finite field of $q$ elements, and let $k\mid q-1$ be a positive integer with
$k\ge2$. Let $f(x)=ax^2+bx+c\in\F_q[x]$ be a quadratic polynomial with $b^2-4ac\ne0$. Suppose
$q>\max\{e^{e^3}, (2k)^6\}$. Then
there is a primitive element $g$ such that $f(g)\in\F_q^{\times k}$.
\end{theorem}

To state our next result, we introduce some notations here. For any positive integer $n$, we let ${\rm
Rad}(n)$ be the product of all distinct primes dividing $n$, and let $W(n)$ be the number of all
square-free positive factors of $n$. Now we state our next result.
\begin{theorem}\label{Thm. B}
Let $\F_q$ be the finite field of $q$ elements, and let $k\mid q-1$ be a positive integer with
$k\ge2$. Let $f(x)=ax^2+bx+c\in\F_q[x]$ be a quadratic polynomial with $b^2-4ac\ne0$. Let $t$ be a
positive factor of $q-1$ with ${\rm Rad}(t)<{\rm Rad}(q-1)$. Suppose that $p_1,\cdots,p_s$ are all
primes dividing $q-1$ but not $t$. Then there is a primitive element $g$ such that
$f(g)\in\F_q^{\times k}$ if
$$q>4k^2W(t)^2\(2+\frac{s-1}{\delta}\)^2,$$
where $\delta=1-\sum_{i=1}^sp_i^{-1}$.
\end{theorem}

As a corollary of Theorem \ref{Thm. B}, we confirm Sun's conjecture and obtain the following result.
\begin{corollary}\label{Coro. A}
Let $\F_q$ be the finite field of $q$ elements. Let $f(x)=ax^2+bx+c\in\F_q[x]$ be a quadratic
polynomial with $b^2-4ac\ne0$.  Then there is a primitive element $g$ such that $f(g)\in\F_q^{\times
2}$ if
$$q\notin\{3,5,7,9, 11,13,19,25,27,31,81,125,121,169\}.$$

\end{corollary}

For example, let $\F_{9}=\F_{3}[T]/(T^2+1)$. Then it is easy to see that
$$\mathcal{P}=\{\alpha T+\beta \mod (T^2+1): \alpha,\beta=\pm1\}$$
is the set of all primitive elements of $\F_{9}$. Let $f(x)=x^2+1$. Then one can verify that
$$\{f(g):g\in\mathcal{P}\}=\mathcal{P}.$$
We will prove Theorem \ref{Thm. A} in Section 2. In Section 3, we will prove Theorem \ref{Thm. B} and
confirm Sun's conjecture.
\maketitle
\section{Proof of Theorem \ref{Thm. A}}
\setcounter{lemma}{0}
\setcounter{theorem}{0}
\setcounter{corollary}{0}
\setcounter{remark}{0}
\setcounter{equation}{0}
\setcounter{conjecture}{0}
We first introduce some notations. Let $\varphi$ be the Euler totient function, and let $\mu$ be the
M\"{o}bius function. Let $\F_q$ be the finite field of $q$ elements. We use the symbol
$\widehat{\F_q^{\times}}$ to denote the group of all multiplicative characters over $\F_q$, and we let
$\ve$ denote the trivial character. In addition, for any $\chi\in\widehat{\F_q^{\times}}$ we define
$\chi(0)=0$. Also, for any positive divisor $d$ of $q-1$ we define
$$G_d:=\left\{\chi\in\widehat{\F_q^{\times}}: \chi\ \text{is of order d}\right\}.$$

Given any positive divisor $t$ of $q-1$, we say that an element $x\in\F_q$ is {\it $t$-free} if $x\ne
0$ and $x=y^d$ (where $y\in\F_q$, $d>0$ and $d\mid t$) implies $d=1$. It is easy to verify that an
element $g\in\F_q$ is a primitive element of $\F_q$ if and only if $g$ is $(q-1)$-free. Moreover, by
the results obtained by Cohen and Huczynska \cite{CH2003,CH2010}, for any $x\in\F_q$ we have
\begin{equation}\label{Eq. Characteristic Function of e-free elements}
\frac{\varphi(t)}{t}\sum_{d\mid t}\frac{\mu(d)}{\varphi(d)}\sum_{\chi\in G_d}\chi(x)=\begin{cases}
1
&\mbox{if}\ x\ \text{is $t$-free},
\\0
&\mbox{otherwise.}\end{cases}
\end{equation}
On the other hand, for any positive divisor $k$ of $q-1$, fix a character
$\chi_k\in\widehat{\F_q^{\times}}$ of order $k$. Then for any $x\in\F_q$ we have
\begin{equation}\label{Eq. Characteristic Function of k-th power elements}
\frac{1}{k}\sum_{i=0}^{k-1}\chi_k^i(x)=\begin{cases}
1
&\mbox{if}\ x\in\F_q^{\times k},
\\0
&\mbox{otherwise.}\end{cases}
\end{equation}
We also need the following result which is known as the Weil Theorem (cf. \cite[Theorem 5.41]{LN}).
\begin{lemma}\label{Lem. Weil's theorem}{\rm (Weil's Theorem)}
Let $\psi\in\widehat{\F_q^{\times}}$ be a character of order $m>1$, and let $f(x)\in\F_q[x]$ be a
monic polynomial which is not of the form $g(x)^m$ for any $g(x)\in\F_q[x]$. Let $r$ be the number of
distinct roots of $f(x)$ in the algebraic closure of $\F_q$. Then for any $a\in\F_q$ we have
\begin{equation}\label{Eq. Weil's Bound}
\left|\sum_{x\in\F_q}\psi(af(x))\right|\le (r-1)q^{1/2}.
\end{equation}
\end{lemma}
For any positive integer $n$, let $\omega(n)$ denote the number of distinct prime factors of $n$.
Robin \cite[Theorem 11]{R} obtained the following upper bound on $\omega(n)$.
\begin{lemma}\label{Lem. Upper bound of w(n)}{\rm (Robin)}
For any positive integer $n\ge3$, we have
\begin{equation}\label{Eq. upper bound of w(n)}
\omega(n)\le 1.38402\log n/\log\log n.
\end{equation}
\end{lemma}
Following the argument in \cite[(4.1)]{CET2015} for all $q\ge 17$ we have
\begin{equation}\label{Eq. w(q-1)}
\omega(q-1)\le 1.38402\log q/\log\log q.
\end{equation}
Recall that $W(q-1)$ denote the number of all positive square-free divisors of $q-1$. It is easy to
see that
\begin{equation}\label{Eq. upper bound of bigW (q-1)}
W(q-1)=2^{\omega(q-1)}\le q^{0.96/\log\log q}.
\end{equation}
Now we prove our first result.

\noindent{\bf Proof of Theorem \ref{Thm. A}.} Let $\F_q$ be the finite field of $q$ elements, and let
$f(x)=ax^2+bx+c$ be a quadratic polynomial in $\F_q[x]$ with $b^2-4ac\ne0$.
For any positive divisor $k$ of $q-1$ with $k\ge2$, we define
\begin{equation*}\label{Eq. Definition of N(q-1,k)}
N(q-1,k):=\left|\{x\in\F_q: x\ \text{is a primitive element and}\ f(x)\in\F_q^{\times k}\}\right|,
\end{equation*}
where $|S|$ denotes the cardinality of a set $S$. By (\ref{Eq. Characteristic Function of e-free
elements}) and (\ref{Eq. Characteristic Function of k-th power elements}), we have
\begin{equation}\label{Eq. N(q-1,k) in the proof of Thm. A}
N(q-1,k)=\frac{\varphi(q-1)}{k(q-1)}\sum_{i=0}^{k-1}\sum_{d\mid
q-1}\frac{\mu(d)}{\varphi(d)}\sum_{\chi\in G_d}S_i(\chi),
\end{equation}
where
$$S_i(\chi):=\sum_{x\in\F_q}\chi(x)\chi_k^i(f(x)).$$
{\bf Case I.} $i=0$ and $\chi=\ve$ is trivial.

In this case, we clearly have
\begin{equation}\label{Eq. A positive part in the proof of Thm. A}
S_0(\ve)\ge q-3.
\end{equation}
{\bf Case II. } $0<i\le k-1$ or the order $d$ of $\chi$ is greater than $1$.

In this case, let $m$ be the least common multiple of $d$ and $k$. Set $dc_1=m$ and $kc_2=m$. Clearly
we can find a character $\chi_m\in\widehat{\F_q^{\times}}$ of order $m$ and choose positive integers
$r_1,r_2$ with $(r_1,d)=1$, $(r_2,k)=1$ such that $\chi_d=\chi_m^{c_1r_1}$ and
$\chi_k=\chi_m^{c_2r_2}$. We claim that by the suitable choice of $\chi_m$, $r_1$ and $r_2$,
the polynomial $x^{c_1r_1}f(x)^{ic_2r_2}\ne g(x)^{m}$ for any $g(x)\in\F_q[x]$.

Suppose first that $d>1$ and that $x^{c_1r_1}f(x)^{ic_2r_2}=g(x)^{m}$ for some $g(x)\in\F_q[x]$. Let
$\deg(g)$ denote the degree of $g(x)$, and let $\ord_x(g):=\max\{n\in\Z: x^n\mid g(x)\}$. Then we have
\begin{equation}\label{Eq. not a m-th power A}
c_1r_1+2ic_2r_2=m\cdot\deg(g),
\end{equation}
\begin{equation}\label{Eq. not a m-th power B}
c_1r_1+ic_2r_2\cdot\ord_x(f)=m\cdot\ord_x(g).
\end{equation}
This gives
$$ic_2r_2(2-\ord_x(f))=kc_2(\deg(g)-\ord_x(g)).$$
Noting that $(r_2,k)=1$, we obtain $k\mid i(2-\ord_x(f))$.
As $b^2-4ac\ne0$, we have $\ord_xf\le 1$ and $\deg(g)>\ord_xg$. Combining this with $0\le i\le k-1$,
the above equality implies that $\ord_xf=0$, $k=2i>0$ and $r_2=\deg(g)-\ord_x(g)$. By (\ref{Eq. not a
m-th power B}) we get $r_1=d\cdot\ord_x(g)$. As $(r_1,d)=1$, we have $d=1$. This contradicts $d>1$.

Suppose now $d=1$ and $0<i\le k-1$. We can choose $m=k$, $\chi_m=\chi_k$, $c_1=k$, $r_1=r_2=1$ and
$c_2=1$. We also claim that $x^{c_1r_1}f(x)^{ic_2r_2}\ne g(x)^{m}$ for any $g(x)\in\F_q[x]$. In fact,
suppose that
$$x^kf(x)^i=g(x)^k,$$
for some $g(x)\in\F_q[x]$. Then we have
$$f(x)^i=\(\frac{g(x)}{x}\)^k=h(x)^k$$
for some $h(x)\in\F_q[x]$. As $0<i\le k-1$, we obtain $k=2i$ and $\deg(h)=1$. This is a contradiction
since $f(x)$ has no multiple roots. Our claim therefore holds. As
$$S_i(\chi)=\sum_{x\in\F_q}\chi_m(x^{c_1r_1}f(x)^{ic_2r_2}),$$
By Lemma \ref{Lem. Weil's theorem} if $0<i\le k-1$ or the order $d$ of $\chi$ is greater than $1$,
then
\begin{equation}\label{Eq. C in the proof of Thm. A}
\left|S_i(\chi)\right|\le 2q^{1/2}.
\end{equation}
Combining (\ref{Eq. N(q-1,k) in the proof of Thm. A}) with (\ref{Eq. A positive part in the proof of
Thm. A}) and (\ref{Eq. C in the proof of Thm. A}), when $q\ge 5$ we have
\begin{align*}
\frac{k(q-1)}{\varphi(q-1)}N(q-1,k)&\ge q-3-2\sqrt{q}(W(q-1)-1)-(k-1)2\sqrt{q}W(q-1)\\
&=q+2\sqrt{q}-3-2kW(q-1)\sqrt{q}\\
&\ge q-2kW(q-1)\sqrt{q}.
\end{align*}
By (\ref{Eq. upper bound of bigW (q-1)}) when $q>\max\{e^{e^{3}},(2k)^6\}$ we have
$$q-2kW(q-1)\sqrt{q}\ge\sqrt{q}(q^{0.5}-2kq^{0.32})>0.$$
This completes the proof of Theorem \ref{Thm. A}.\qed
\maketitle
\section{Proof of Theorem \ref{Thm. B}}
\setcounter{lemma}{0}
\setcounter{theorem}{0}
\setcounter{corollary}{0}
\setcounter{remark}{0}
\setcounter{equation}{0}
\setcounter{conjecture}{0}
In this section, we adopt notations defined in Section 2. Let $\F_q$ be the finite field of $q$
elements, and let $k\ge2$ be a positive divisor of $q-1$. Let $f(x)=ax^2+bx+c$ be a quadratic
polynomial in $\F_q[x]$ with $b^2-4ac\ne0$. For any positive factor $t$ of $q-1$, we define
\begin{equation}\label{Eq. Definition of N(e,k)}
N(t,k):=\left|\{x\in\F_q: x\ \text{is $t$-free and}\ f(x)\in\F_q^{\times k}\}\right|.
\end{equation}
For any positive integer $n$, recall that ${\rm Rad}(n)$ denotes the product of all distinct primes
dividing $n$.

With the method used in \cite[Lemma 1]{CET2015B}, we obtain the following result.
\begin{lemma}\label{Lem. improvement of the bound}
Let $t$ be a divisor of $q-1$. Suppose ${\rm Rad}(t)<{\rm Rad}(q-1)$. Let $p_1,\cdots,p_s$ be the
distinct primes which divide $q-1$ but not $t$, and set $\delta=1-\sum_{i=1}^sp_i^{-1}$. Then we have
the following result:
\begin{equation}\label{Eq. A in Lem improvement of the bounds}
N(q-1,k)\ge\sum_{i=1}^sN(p_it,k)-(s-1)N(t,k).
\end{equation}
Hence we have
\begin{equation}\label{Eq. improvement of the bound}
N(q-1,k)\ge\sum_{i=1}^s\(N(p_it,k)-\(1-\frac{1}{p_i}\)N(t,k)\)+\delta N(t,k).
\end{equation}
\end{lemma}
\begin{proof}
For any $x\in\F_q$, it is clear that $x$ contributes $+1$ to the right hand side of (\ref{Eq. A in Lem
improvement of the bounds}) if $x$ is a primitive element and $f(x)\in\F_q^{\times k}$. Suppose that
$x$ is not a primitive element or $f(x)\not\in\F_q^{\times k}$. Then clearly $x$ contributes $0$ or a
negative integer to the right hand side of (\ref{Eq. A in Lem improvement of the bounds}). Hence
inequality (\ref{Eq. A in Lem improvement of the bounds}) holds and (\ref{Eq. improvement of the
bound}) is clearly a consequence of (\ref{Eq. A in Lem improvement of the bounds}).

This completes the proof.
\end{proof}
We also need the following result.
\begin{lemma}\label{Lem. N(e,k)}
Suppose $q\ge5$ is a prime power, and let $t$ be a positive divisor of $q-1$. Then we have
\begin{equation}\label{Eq. N(e,k)}
N(t,k)\ge\frac{\varphi(t)}{kt}\(q-2kW(t)q^{1/2}\).
\end{equation}
Moreover, for any prime $p$ dividing $q-1$ but not $t$ we have
\begin{equation}\label{Eq. N(pe,k)-N(e,k)}
\left|N(pt,k)-\(1-\frac{1}{p}\)N(t,k)\right|\le \frac{\varphi(pt)}{pt}2W(t)q^{1/2}.
\end{equation}
\end{lemma}
\begin{proof}
{\rm (i)} By (\ref{Eq. Characteristic Function of e-free elements}) and (\ref{Eq. Characteristic
Function of k-th power elements}) we have
$$N(t,k)=\frac{\varphi(t)}{kt}\sum_{i=0}^{k-1}\sum_{d\mid t}\frac{\mu(d)}{\varphi(d)}\sum_{\chi\in
G_d}S_i(\chi),$$
where
$$S_i(\chi)=\sum_{x\in\F_q}\chi(x)\chi_k^i(f(x)).$$
With essentially the same method in the proof of Theorem \ref{Thm. A}, one can easily verify (\ref{Eq.
N(e,k)}).

{\rm (ii)} As
$$N(pt,k)=\frac{\varphi(pt)}{kpt}\sum_{i=0}^{k-1}\sum_{d'\mid
pt}\frac{\mu(d')}{\varphi(d')}\sum_{\chi\in G_{d'}}S_i(\chi),$$
one can verify that
\begin{equation}\label{Eq. in the proof of lemma N(e,k)}
N(pt,k)-\(1-\frac{1}{p}\)N(t,k)=\frac{\varphi(pt)}{kpt}\sum_{i=0}^{k-1}\sum_{d\mid
t}\frac{\mu(pd)}{\varphi(pd)}\sum_{\chi\in G_{pd}}S_i(\chi).
\end{equation}
With essentially the same method in the proof of Theorem \ref{Thm. A}, it is easy to verify that
$$\left|S_i(\chi)\right|\le2q^{1/2}$$
for all $S_i(\chi)$ which appear in (\ref{Eq. in the proof of lemma N(e,k)}).
Hence (\ref{Eq. N(pe,k)-N(e,k)}) holds.

This completes the proof.
\end{proof}
Now we are in a position to prove Theorem \ref{Thm. B}.

\noindent{\bf Proof of Theorem \ref{Thm. B}.} Combining Lemma \ref{Lem. improvement of the bound} with
Lemma \ref{Lem. N(e,k)}, we have
\begin{align*}
\frac{kt}{\delta\varphi(t)}N(q-1,k)&\ge
q-2kW(t)q^{1/2}-2kW(t)q^{1/2}\frac{\sum_{i=1}^s\(1-\frac{1}{p_i}\)}{\delta}\\
&=\sqrt{q}\(\sqrt{q}-2kW(t)\(2+\frac{s-1}{\delta}\)\)>0.
\end{align*}
This completes the proof.\qed

Now we turn to Sun's conjecture.

\noindent{\bf Proof of Corollary \ref{Coro. A}.} Let $P_i$ denote the $i$-th prime. Note that
$\prod_{i=1}^{10}P_i=6469693230>e^{e^3}$. When $k=2$, if the number of distinct prime factors of $q-1$
is not less than $10$, then the corollary holds by Theorem \ref{Thm. A}. Hence it is sufficient to
consider the cases that the number of distinct prime factor prime factors of $q-1$ is less than $10$.
Note that $\delta\geq1-\Sigma_{i=w(q-1)-s+1}^{w(q-1)}\frac{1}{P_{i}}$. Hence we can get a upper bound
for $q$ such that the desired result holds whenever $q$ is greater than this bound. On the other hand,
$q>q-1\geq \prod_{i=1}^{w(q-1)}P_{i}$ makes it possible to reduce the cases furthermore. The detailed
data will be listed in Table 1 below.

As a result, it suffices to check the cases $q=3^2,3^3,3^4,5^2,5^3,7^2,11^2,13^2$ and those odd primes
less than 523. Moreover, Booker, Cohen, Sutherland and Trudgian \cite{BCST2019} have proved that there
exists a primitive element $g$ such that $f(g)$ is also a primitive element whenever $q>211$. Note
that every primitive element must be in $\F_q^{\times}\setminus\F_q^{\times2}$. Suppose
$a\not\in\F_q^{\times 2}$. Then one can get the desired result by applying the above result to
$f(x)/a$. Hence it is sufficient to consider the cases that $f(x)=x^2+bx+c$ with $b^2-4c\ne0$ when
$q>211$. After computation, we can verify corollary \ref{Coro. A}. \qed

\begin{table}
	\centering
	\caption{data}
	\begin{tabular}
		{>{\columncolor{white}}rccccc}
		\toprule[1pt]
		\rowcolor[gray]{0.9}	&$w(q-1)$ &$\prod_{i=1}^{w(q-1)}P_{i}$ &$s=1$ &$s=2$ &$s=3$\\
		\midrule
		&1 &2 &32 & &\\
		&2 &6 &128 &128 &\\
		&3 &30 &512 &265 &\\
		&4 &210 &2048 &901 &523\\
		&5 &2310 & &3384 &498\\
		&6 &30030 & &13114 &\\
		&7 &510510 & &51725 &\\
		&8 &9699690 & &204828 &\\
		&9 &223092870 & &814305 &\\
		\bottomrule[1pt]
	\end{tabular}
\end{table}
\section{Acknowledgment}

\acknowledgment\ This research was supported by the National Natural Science Foundation of
China (Grant No. 11971222). The first author was supported by NUPTSF (Grant No. NY220159).

\end{document}